\documentclass{amsproc}
\usepackage{latexsym}
\usepackage{amssymb}
\usepackage{amsmath}
\usepackage{amscd}
\usepackage[all]{xy}
\usepackage{graphicx}
\usepackage{float}

\newtheorem{thm}{Theorem}
\newtheorem{pro}{Proposition}

\author{Lisa Nilsson \& Mikael Passare}
\address{Department of Mathematics,
Stockholm University,
\hfill\break
\phantom{\hskip.54cm} SE-106 91  Stockholm, Sweden}
\email{lisa@math.su.se, passare@math.su.se}

\newtheorem{dfn}{Definition}

\newcommand{\Log}{\operatorname{Log}}

\newcommand{\Arg}{\operatorname{Arg}}
\newcommand{\Exp}{\operatorname{Exp}}

\def\Im{{\rm Im\, }}
\def\Re{{\rm Re\, }}

\def\bee{\begin{equation*}}
\def\eee{\end{equation*}}

\def\be{\begin{equation}}
\def\ee{\end{equation}}

\title{Mellin transforms of multivariate rational functions}
\date{\today}

\begin{document}

\begin{abstract}

This paper deals with Mellin transforms of rational functions $g/f$ in several variables. We prove that 
the polar set of such a Mellin transform consists of finitely many families of parallel hyperplanes, with all planes in each such family being
integral translates of a specific facial hyperplane of the Newton polytope of the denominator $f$. The Mellin transform is naturally related to 
the so called coamoeba $\mathcal{A}'_f:=\text{Arg}\,(Z_f)$, where
$Z_f$ is the  zero locus of $f$ and $\text{Arg}$ denotes the mapping that takes each coordinate to its argument. In fact, each connected component of
the complement of the coamoeba $\mathcal{A}'_f$ gives rise to a different Mellin transform. The dependence of the Mellin transform on the coefficients
of $f$, and the relation to the theory of $A$-hypergeometric functions is also discussed in the paper. 
\end{abstract}

\maketitle

\section{Introduction}

\noindent The Mellin transform $M_h$ of a locally integrable function $h$ on the positive real axis is defined by the formula 
\be\label{mellt} M_h(s)=\int_0^\infty h(z)\,z^s\,\frac{dz}{z}\,,\ee

\noindent provided the integral converges. Here  $s$ is a complex variable $s=\sigma+it$. 
The Mellin transform is closely related to the Fourier--Laplace transform via an exponential change of variables. More precisely, the value of $M_h(s)$ is 
equal to the Fourier--Laplace transform of the function $x\mapsto h(e^{-x})$ evaluated at the point $-is$. 

\smallskip
In this paper we consider Mellin transforms of rational functions $h=g/f,$ where $g$ and $f$ are polynomials. Since the general case is easily settled once we have fully investigated the special case where $g\equiv1$, and since this will simplify our notation  and therefore clarify our argument, we shall focus mainly on the case $h=1/f$.  

\smallskip
Let us start by considering the one-variable situation. Given a polynomial $$f(z)=a_0+a_1z+\ldots+a_mz^m$$ we assume for the moment that its coefficients 
$a_0,\ldots,a_m$ are positive numbers.  Then the integral (\ref{mellt}) with $h=1/f$ converges and defines an analytic function in the vertical strip $0<\sigma<m$. 

\smallskip
One can in fact make a meromorphic continuation of this Mellin transform and write it as
\begin{equation}\label{psi}M_{1/f}(s)=\Phi(s)\Gamma(s)\Gamma(m-s)\,,\end{equation}

\noindent  where $\Phi$ is an entire function. 
To see this, let us first look at the case of a simple fraction $1/(a+bx)$.
In this case one has the explicit formula
$$M_{1/f}(s)=a^{s-1}b^{-s}\Gamma(s)\Gamma(1-s),$$
 
\noindent which can be easily established for instance by means of a residue computation.
Now, considering a general product 
$$f(z)=\prod_{j=1}^m(\alpha_j+z),$$
one can decompose $1/f$ into a sum of simple fractions and hence immediately deduce that 
its Mellin transform will be of the form $\Psi(s) \Gamma(s)\Gamma(1-s)$,
for some entire function $\Psi$. In fact, a straightforward residue calculation shows that 
$$\Psi(s)=-e^{-\pi is}\sum \mathrm{res}[z^{s-1}/f(z)]\,,$$

\noindent and by the theorem on the total sum of residues it then follows that
$$\Psi(1)=\Psi(2)=\ldots=\Psi(m-1)=0\,.$$
This means that we obtain formula (\ref{psi}) with $\Phi(s)=\Psi(s)/\bigl[(1-s)\cdots(m-1-s)\bigr]$. 

\smallskip
We have thus found that all the poles of the meromorphic continuation are located at the two integer sequences $0,-1,-2,\ldots $
and $m+1,m+2,\ldots$ emanating from the end points of the interval $[0,m]$. Notice that this interval is the Newton polytope of 
our one-variable polynomial $f$. 

\smallskip
As a matter of fact, in the above discussion we did not actually need to assume that the coefficients $a_0,a_1,\ldots,a_m$ be positive. 
A necessary and sufficient condition for the argument to work, and in particular for the integral to converge, is that $f(z)\neq 0$ 
for all real positive values of $z$. Another way of formulating this latter condition is that $f$ should have  no roots 
with argument zero.

\smallskip
We now turn to the multidimensional case, and we begin looking at a simple example with the denominator $f$ being an
affine linear polynomial.

\bigskip
\noindent\textbf{Example~1.} 
\noindent Consider the polynomial $f(z)=1+z_1+z_2$. The Mellin transform of the corresponding rational function $1/f$ is then given by the integral

$$\int_0^\infty\!\!\!\int_0^\infty\!\frac{z_1^{s_1}z_2^{s_2}}{1+z_1+z_2}\frac{dz_1dz_2}{z_1z_2}=\int_0^\infty\!\!\!\int_0^\infty\!\!\!\int_0^\infty\! z_1^{s_1-1}z_2^{s_2-1}\,e^{-v(1+z_1+z_2)}dz_1dz_2dv\,,$$

\bigskip
\noindent which after the coordinate change $t=vz_1$, $u= vz_2$ becomes 

$$\int_0^\infty\!\!\! t^{s_1-1}e^{-t}dt\int_0^\infty \!\!\!u^{s_2-1}e^{-u}du\int_0^\infty \!\!\!v^{-s_1-s_2}\,e^{-v}dv = 
\Gamma(s_1)\Gamma(s_2)\Gamma(1-s_1-s_2)\,.$$

\bigskip
\smallskip 
We shall see in this paper that the fact that the poles of the Mellin transform are determined by a product of $\Gamma$-functions is not unique for the special cases we have
considered so far. In fact, the Mellin transform of a rational function in any number of variables will turn out to be always a product of $\Gamma$-functions 
in linear arguments, multiplied by some entire function, so that the only poles of the Mellin transform are the poles of the $\Gamma$-functions. 
Moreover, the configuration of polar hyperplanes is governed by the Newton polytope of the denominator polynomial, with one family of parallel hyperplanes emanating from each facet of the Newton polytope. Many of the results have been announced previously in \cite{N}. Let us note in passing that a similar phenomenon can be observed also for Mellin transforms of more general meromorphic functions, with transcendental denominators. As an illustration of this we recall the classical formulas
$$M_{1/e^z}=\Gamma(s)\qquad \text{and}  \qquad 
M_{1/(e^z-1)}=\zeta(s)\Gamma(s)=(s-1)\zeta(s)\Gamma(s-1)\,,$$
where $\zeta$ denotes the Riemann zeta function.
\bigskip

\section{Newton polytopes and (co)amoebas}

\noindent Throughout this paper $f$ will denote a complex Laurent polynomial 

\begin{equation}\label{polyn}f(z)=\sum_{\alpha\in A}a_\alpha z^\alpha,\quad a_\alpha\in\mathbb{C}_*\,,\end{equation}

\smallskip

\noindent where $A\subset\mathbb{Z}^n$ is a finite subset and
$\mathbb{C}_*$ denotes the punctured complex plane $\mathbb{C}\setminus{\{0\}}$. Here we use the standard notation $z^\alpha=z_1^{\alpha_1}\cdots z_n^{\alpha_n}$ for $z=(z_1,\ldots,z_n)\in\mathbb{C}_*^n$.
\smallskip

The Newton polytope $\Delta_f$ of the polynomial $f$ is defined to be the convex hull of $A$ in $\mathbb{R}^n$. We shall primarily be interested in the case where $\Delta_f$
has a nonempty interior. Like any other polytope, 
the Newton polytope $\Delta_f$ may be alternatively viewed as the intersection of a finite number of halfspaces: 

\begin{equation}\label{newtonpol}\Delta_f=\bigcap_{k=1}^N\bigl\{\sigma\in\mathbb{R}^n\,;\,\langle\mu_k,\sigma\rangle\geq\nu_k\bigr\}\,,\end{equation}

\noindent where the $\mu_k\in\mathbb{Z}^n$ are primitive integer vectors in the inward normal direction of the facets of  $\Delta_f$, and the $\nu_k\in\mathbb{Z}$ are integers. 
\smallskip

In general we will let $\Gamma$ denote a face of the Newton polytope of arbitrary dimension, $0\leq\mathrm{dim}(\Gamma)\leq\mathrm{dim}(\Delta_f)$, 
and we define the relative interior $\mathrm{relint}(\Gamma)$ of such a face to be the interior of $\Gamma$ viewed as a subset of the lowest dimensional hyperplane 
containing it. For each face $\Gamma$ we also introduce the corresponding truncated polynomial 

$$f_\Gamma=\sum_{\alpha\in\Gamma}a_\alpha z^\alpha\,,$$

\noindent consisting of those monomials from the original polynomial $f$ whose exponents are contained in the face $\Gamma$ of the Newton polytope $\Delta_f$.
\smallskip

The amoeba $\mathcal{A}_f$ and the coamoeba $\mathcal{A}'_f$ of a polynomial $f$ are defined to be the images of the zero set $Z_f=\{ z\in\mathbb{C}^n_*\,;\ f(z)=0\,\}$
under the real and imaginary parts, $\Log$ and $\Arg$ respectively, of the coordinatewise complex logarithm mapping. More precisely, one has
$$A_f=\Log(Z_f)\qquad \text{and}  \qquad A'_f=\Arg(Z_f)\,,$$
where $\Log(z)=(\log|z_1|,\ldots,\log|z_n|)$ and $\Arg(z)=(\arg(z_1),\ldots,\arg(z_n))$. 
\smallskip

Writing $w=x+i\theta\in\mathbb{C}^n$ and $z=\Exp(w)=(\exp(w_1),\ldots,\exp(w_n))$, 
one obtains the identities $x=\Log(z)=\Re(w)$ and $\theta=\Arg(z)=\Im(w)$,
as illustrated in the following picture.

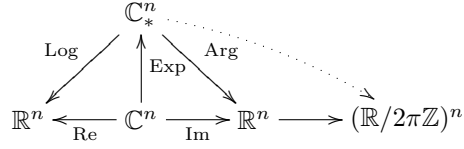
\begin{figure}[H]
\centering

$$\xymatrix{  
                    &{\mathbb{C}^n_*}\ar[dl]_{\Log}  \ar[dr]^{\Arg}  \ar@/^/@{.>}[drr] & \\  
            {\mathbb{R}^n} &  {\mathbb{C}^n}\ar[u]_{\Exp} \ar[l]^{\mathrm{Re}} \ar[r]_{\mathrm{Im}}  &   {\mathbb{R}^n} \ar[r] & {(\mathbb{R}/2\pi\mathbb{Z})^n}}$$                              
\vspace{.5cm}
\caption{Real and imaginary parts of the complex logarithm mapping}\label{avb}

\end{figure}

The amoeba $\mathcal{A}_f$ is a subset in $\mathbb{R}^n$, whereas the coamoeba $\mathcal{A}'_f$ can be viewed as being located either in the $n$-dimensional torus $(\mathbb{R}/2\pi\mathbb{Z})^n$ or as a multiply periodic subset of $\mathbb{R}^n$. This reflects the multivaluedness of the argument mapping.
\smallskip

For brevity of notation we denote the amoeba and the coamoeba of a truncated polynomial $f_\Gamma$ by $\mathcal{A}_\Gamma$ and $\mathcal{A}'_\Gamma$. 
\bigskip

 \section{Mellin transforms of rational functions}

\noindent The natural generalization to several variables of the standard Mellin transform of a rational function $1/f$ is given by  
the integral

\be\label{mellint} M_{1/f}(s)=\int_{\mathbb{R}_+^n}\frac{ z^{s}}{{f(z)}}\frac{dz_1\wedge\ldots\wedge dz_n}{z_1\cdots z_n}=\int_{\mathbb{R}^n} \frac{e^{\langle s,x\rangle}}{f(e^x)}\,dx_1\wedge\ldots\wedge x_n,\ee

\noindent where $\mathbb{R}^n_+=(0,\infty)^n$ denotes the positive orthant in $\mathbb{R}^n$. In order for such an integral to converge 
one has to make some assumptions about the exponent vector $s$ and also about the denominator $f$. It turns out that it is not
enough to demand only that $f$ be non-vanishing on $\mathbb{R}^n_+$.
\smallskip

\begin{dfn} A polynomial $f$ is said to be \emph{completely non-vanishing} on a set $X$ if for all faces $\Gamma$ of the Newton polytope $\Delta_f$ the truncated 
polynomial $f_\Gamma$ has no zeros on $X$. In particular, the polynomial $f$ itself does not vanish on $X$.
\end{dfn}\smallskip

\noindent{\bf Remark.}\ This concept of completely non-vanishing polynomials is closely related to the notion of quasielliptic polynomials discussed in \cite{ET}.
\smallskip

\begin{thm}\label{expgro}
If the polynomial $f$ is completely non-vanishing on the positive orthant $\mathbb{R}^n_+$ then the integral (\ref{mellint}) converges and defines an 
analytic function in the tube domain $\bigl\{\,s\in\mathbb{C}^n\,;\ \text{\rm Re}\,s=\sigma\in\text{\rm int}\,\Delta_f\,\bigr\}$. 
\end{thm}

\begin{proof} It will suffice to prove that for any given $s$ with $\sigma\in\text{\rm int}\,\Delta_f$ there are
positive constants $c$, $k>0$ such that

\begin{equation}\label{olik} \big|f(e^x)e^{-\langle s,x\rangle}\big| = \big|f(e^x)\big|\,e^{-\langle\sigma,x\rangle}\ge c\,e^{k|x|}\,,\quad x\in\mathbb{R}^n\,.\end{equation}

\smallskip

The proof is by induction on the dimension $n$. The case $n=1$ is easy. Let $\alpha$ and $\beta$ with $\alpha<\beta$ be the two endpoints of $\Delta_f$. Then for sufficiently large negative $x$ one has
$$
\big|f(e^x)\big|\,e^{-\sigma\cdot x}\ge \frac{1}{2}\,|a_\alpha|\,e^{(\sigma-\alpha)|x|}\,,
$$
and for sufficiently large positive $x$
$$
\big|f(e^x)\big|\,e^{-\sigma\cdot x}\ge \frac{1}{2}\,|a_\beta|\,e^{(\beta-\sigma)|x|}\,.
$$
Now make the induction hypothesis that the inequality (\ref{olik}) holds for dimensions $\le n-1$, and consider a polynomial $f$ of $n$ variables.
For each face $\Gamma$ of $\Delta_f$, with $0\le \dim\Gamma\le n-1$, the given point $\sigma$ can be expressed as a convex combination
$$
\sigma=\lambda\,\sigma_\Gamma+(1-\lambda)\tau_\Gamma\,,
$$
where $\sigma_\Gamma\in\text{\rm relint}(\Gamma)$ and $\tau_\Gamma\in\text{\rm relint}\bigl(\text{\rm conv}(A\setminus \Gamma)\bigr)$.
Fix a choice of such a point $\sigma_\Gamma$ in each face $\Gamma$, and consider for each $\Gamma$  the new convex polytope
$$
\Delta_\Gamma=\text{\rm conv}\bigl((A\setminus\Gamma)\cup\sigma_\Gamma\bigr)\,.$$
Notice that when $\dim\Gamma=0$, that is, when $\Gamma$ is a vertex of $\Delta_f$, one has $\Delta_\Gamma=\Delta_f$. Notice also that the original point $\sigma$ belongs to each $\Delta_\Gamma$.
\smallskip

Let $\widetilde C_\Gamma$ be the outer normal cone to $\Delta_\Gamma$ with vertex at $\sigma_\Gamma$:
\be\label{ineq}
\widetilde C_\Gamma=\bigl\{\,x\in\mathbb{R}^n\,\,|\,\ \langle\xi-\sigma_\Gamma,x-\sigma_\Gamma\rangle\le0\,,\  \forall\,\xi\in\Delta_\Gamma\,\bigr\}\,.
\ee
All these cones $\widetilde C_\Gamma$ are of full dimension $n$ and together they almost cover the entire space $\mathbb{R}^n$. More precisely, the complement
$$
\mathbb{R}^n\setminus\Bigl(\bigcup_\Gamma\widetilde C_\Gamma\Bigr)
$$
is a bounded subset of $\mathbb{R}^n$. Then one can let $C_\Gamma$ be a slightly smaller closed convex cone, still with vertex at $\sigma_\Gamma$, such that $C_\Gamma\setminus\sigma_\Gamma$ is contained in the interior of $\widetilde C_\Gamma$, and such that the complement of the union $\cup_\Gamma C_\Gamma$ is still a bounded set. 
Notice that for $x\in C_\Gamma\setminus\sigma_\Gamma$ the inequality in \eqref{ineq} will be strict, and we may in fact assume this to be true uniformly.
\smallskip

We now observe that it is enough to prove the estimate (\ref{olik}) for $x\in C_\Gamma$. Actually, it suffices to do it for 
$x\in C_\Gamma\setminus B_R(0)$ for some large ball $B_R(0)$. From the induction hypothesis we conclude that there are constants $c_\Gamma$ such that
$$
\Big|f_\Gamma(e^x)e^{-\langle\sigma_\Gamma,x\rangle}\Big|\ge c_\Gamma>0\,,\quad x\in\mathbb{R}^n\,.
$$
Indeed, $f_\Gamma(e^x)$ is a function depending on fewer variables than $n$, since it is homogeneous in directions orthogonal to $\Gamma$, and $\sigma_\Gamma\in\text{\rm relint}(\Delta_{f_\Gamma})$.
\smallskip

For each face $\Gamma$ let $g_\Gamma(z)$ be the function containing all the monomials not on $\Gamma$ so that $f_\Gamma+g_\Gamma=f$.
Now we use the decomposition $f=f_\Gamma +g_\Gamma$ so that one obtains
\begin{equation}\label{p1}
f(e^x)e^{-\langle\sigma,x\rangle}=e^{\langle\sigma_\Gamma-\sigma,x\rangle}\bigl(f_\Gamma(e^x)e^{-\langle\sigma_\Gamma,x\rangle}+g_\Gamma(e^x)e^{-\langle\sigma_\Gamma,x\rangle}\bigr)\,.
\end{equation}
Take $x\in C_\Gamma$ and write $x=\sigma_\Gamma+y$.  Recall that $\sigma\in\Delta_f$. The first factor $e^{\langle\sigma_\Gamma-\sigma,x\rangle}$ can be estimated from below by $c_0e^{k|y|}$ with the positive constants $c_0$ and $k$ given by $c_0=\exp\langle\sigma_\Gamma-\sigma,\sigma_\Gamma\rangle$, and 
$$
k=\min\{\langle\sigma_\Gamma-\sigma,y\rangle\,;\ |y|=1\,,\ \sigma_\Gamma+y\in C_\Gamma\,\}>0\,.
$$ Assuming, which we may, that $|x|>|\sigma_\Gamma|$, and hence that $|x|-|\sigma_\Gamma|\geq|x-\sigma_\Gamma|=y$, we find
$$
e^{\langle\sigma_\Gamma-\sigma,x\rangle}\ge c_1e^{k|x|}\,,
$$
where $c_1=c_0e^{-k|\sigma_\Gamma|}$. 
\smallskip

To finish the proof of the inequality we now only need to bound the expression in brackets in (\ref{p1}) from  below by a positive constant. From the induction hypothesis we have that $$|f_\Gamma e^{-\langle\sigma_\Gamma,x\rangle}|\geq c_\Gamma>0,$$ and it is therefore enough to show that the remainder term $g_\Gamma(e^x) e^{-\langle\sigma_\Gamma,x\rangle}$ stays small, say $<c_\Gamma/2$.
We have the identity
$$
g_\Gamma(e^x)=\sum_{\alpha\in A\setminus\Gamma}a_\alpha\,e^{\langle\alpha,x\rangle}\sum_{\alpha\in A\setminus\Gamma}\tilde a_\alpha\,e^{\langle\alpha,y\rangle}\,.
$$
Since $\alpha\in\Delta_\Gamma$ we have a strictly positive constant
$$
k_\alpha=\min\bigl\{\,\langle\sigma_\Gamma-\alpha,y\rangle\,;\ |y|=1\,,\ \sigma_\Gamma+y\in C_\Gamma\,\bigr\}\,,
$$
and hence
$$
\big|a_\alpha\,e^{\langle\alpha,x\rangle}\big|=\big|\tilde a_\alpha\,e^{\langle\sigma_\Gamma-\alpha,y\rangle}\big|\le|\tilde a_\alpha|\,e^{-k_\alpha\,|y|}\,.
$$ This means that for some large enough $R_0$ one has
$$
\big|g_\Gamma(e^x)\,e^{-\langle\sigma_\Gamma,x\rangle}\big|<c_\Gamma/2\,,\quad\text{\rm whenever}\quad |\sigma_\Gamma+x|\ge R_0\,,\ x\in C_\Gamma\,.
$$
Hence there is an inequality $|f(e^x)\exp\,(-\langle\sigma_\Gamma,x\rangle)|\ge c_\Gamma/2$, and we can conclude that for all $x$ in 
$C_\Gamma\setminus B_R(0)$, for some large ball $B_R(0)$, one has the desired estimate
$$
\Big|f(e^x)e^{-\langle\sigma,x\rangle}\Big|\ge c\,e^{k|x|}\,,\quad x\in\mathbb{R}^n\,,
$$
with $c=c_1c_\Gamma/2$.
\end{proof}
\bigskip

Having thus established the convergence of the integral (\ref{mellint}) defining the Mellin transform, we now turn to the question of finding its analytic
continuation as a meromorphic function of $s$ in the whole complex space $\mathbb{C}^n$. The polar locus of the meromorphic continuation
turns out to be a finite union of families of parallel hyperplanes. The normal directions of these hyperplanes are precisely the vectors $\mu_k$ from the 
representation (\ref{newtonpol}) of the Newton polytope $\Delta_f$.
\smallskip

\begin{thm}\label{mel}
If the polynomial $f$ is completely non-vanishing on the positive orthant $\mathbb{R}^n_+$ and its Newton polytope $\Delta_f$ is of full dimension,
then the Mellin transform $M_{1/f}$ admits a meromorphic continuation 
of the form
\begin{equation}\label{mero}M_{1/f}(s)=\Phi(s)\prod_{k=1}^N\Gamma(\langle\mu_k,s\rangle-\nu_k),\end{equation}
where $\Phi$ is an entire function, and where $\mu_k,\nu_k$ are the same as in equation (\ref{newtonpol}).

\end{thm}
\smallskip

\noindent Before giving the proof of this theorem let us illustrate the idea of the argument by means
of a specific example.
\bigskip

\noindent\textbf{Example 2.}\ Consider the polynomial $f(z)=1+z_2+z_1^2+z_1z_2^2$. It is easy to check that the
representation (\ref{newtonpol}) of its Newton polytope is given by 
$$\{\sigma_1\geq 0\}\cap\{\sigma_1-\sigma_2\geq -1\}\cap\{-2\sigma_1-\sigma_2\geq -4\}\cap\{\sigma_2\geq 0\}\,,$$
so in this case the Newton polygon $\Delta_f$ has four inward normal vectors given by 
$$\mu_1=(1,0)\,,\quad\mu_2=(1,-1)\,,\quad\mu_3=(-2,-1)\,,\quad \text{and}\quad \mu_4=(0,1)\,.$$

\noindent We know from Theorem~1 that the Mellin transform $M_{1/f}$ is holomorphic for all $s=(s_1,s_2)$ whose real part 
$\sigma=(\sigma_1,\sigma_2)$ lies inside the Newton polygon $\Delta_f$. In order to achieve a meromorphic continuation of
$M_{1/f}$ across the left vertical edge of $\Delta_f$ it suffices to perform an integration by parts with respect to $z_1$.
Indeed, this gives us the identity

\begin{equation}\label{exm}M_{1/f}(s)=\frac{1}{s_1}\!\int_0^\infty\!\!\!\!\int_0^\infty\!\!\frac{(2z_1^2+z_1z_2^2)\,z_1^{s_1}z_2^{s_2}}{(1+z_2+z_1^2+z_1z_2^2)^2}\frac{dz_1dz_2}{z_1z_2}\,,\end{equation}
\smallskip

\noindent and we claim that this integral, that is, the Mellin transform multiplied by $s_1$, converges for all $s$ with real part $\sigma$
in the dark triangle on the left in Figure~2. This means that $M_{1/f}$ has been continued meromorphically over the hyperplane $s_1=0$ as desired.
To verify the claim we decompose the integral in (\ref{exm}) into two Mellin type integrals containing the integrands $2z_1^{2+s_1}z_2^{s_2}/f^2$ and 
$z_1^{1+s_1}z_2^{2+s_2}/f^2$
respectively. Since the Newton polygon of the denominator $f^2$ is equal to the original $\Delta_f$ dilated by a factor $2$, we see that the convergence
domains for these two integrals are given by the translated polygons $(-2,0)+2\Delta_f$ and $(-1,-2)+2\Delta_f$ respectively. The sum of the integrals therefore converges
on the intersection of the translated polygons, and this is precisely the dark triangle on the left in Figure~2.

\vspace{-1cm}
\begin{figure}[H]
\centering
\hspace{1cm}
\includegraphics[height=5cm]{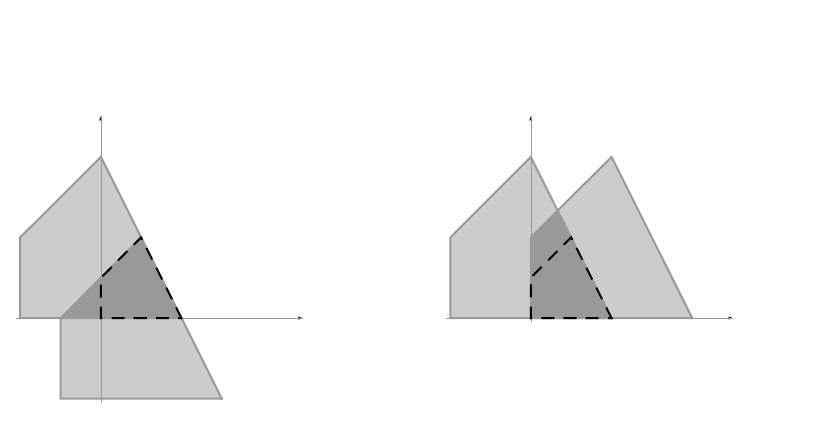}
\caption{The convergence domains (dark) of the integrals after the two cases of integration by parts, given as the intersection of two translated copies of $\Delta_{f^2}
=2\Delta_f$. The dashed polygon is $\Delta_f$.}
\end{figure}

\vspace{-.2cm}
We have thus seen how a meromorphic continuation can be carried out in the horizontal direction, that is, in the direction given by $\mu_1$. Suppose next that we 
wish to obtain a similar mermorphic extension across the upper left edge of $\Delta_f$, the one with normal vector $\mu_2=(1,-1)$. The way to acheive such a ``directional integration by parts" is to suitably introduce a parameter $\lambda$ and then to differentate with respect to $\lambda$. More precisely, we make the coordinate change
$z_1\mapsto\lambda z_1$, $z_2\mapsto\lambda^{-1}z_2$ and obtain 

$$M_{1/f}(s)=\lambda^{1+s_1-s_2}\int_0^\infty\!\!\!\!\int_0^\infty\frac{z_1^{s_1}z_2^{s_2}}{\lambda+z_2+\lambda^3z_1^2+z_1z_2^2}\frac{dz_1dz_2}{z_1z_2}\,.$$ 
\smallskip

\noindent Here the left hand side is obviously independent of $\lambda$. Hence so is the right hand side, and after differentiating and plugging in $\lambda=1$ we
find that

$$0=(1+s_1-s_2)\!\!\int_0^\infty\!\!\!\!\int_0^\infty\frac{z_1^{s_1}z_2^{s_2}}{f(z)}\frac{dz_1dz_2}{z_1z_2}\,-\!\int_0^\infty\!\!\!\!\int_0^\infty\frac{(1+3z_1^2)\,z_1^{s_1}z_2^{s_2}}{f(z)^2}\frac{dz_1dz_2}{z_1z_2}.$$
\smallskip

\noindent This relation can be re-written as

$$M_{1/f}(s)=\frac{1}{1+s_1-s_2}\int_0^\infty\!\!\!\!\int_0^\infty\frac{(1+3z_1^2)\,z_1^{s_1}z_2^{s_2}}{f(z)^2}\frac{dz_1dz_2}{z_1z_2}\,,$$
\smallskip

\noindent and reasoning as above we find that this latter integral converges for all $s$ with real part $\sigma$ in the dark polygon on the right
in Figure~2, thereby yielding a meromorphic continuation across the hyperplane $s_1-s_2=-1$. 
\smallskip

This method of repeatedly performing integration by parts in all the directions $\mu_k$, by using the corresponding coordinate changes $z_j\mapsto\lambda^{\mu_{kj}}z_j$,
is the basis for our proof of Theorem~2, and it gives a global meromorphic continuation of the original Mellin integral. For our special example, the picture below indicates the 
full set of polar hyperplanes, going out in all directions from the Newton polytope $\Delta_f$.

\begin{figure}[H]
\centering
\includegraphics[height=5cm]{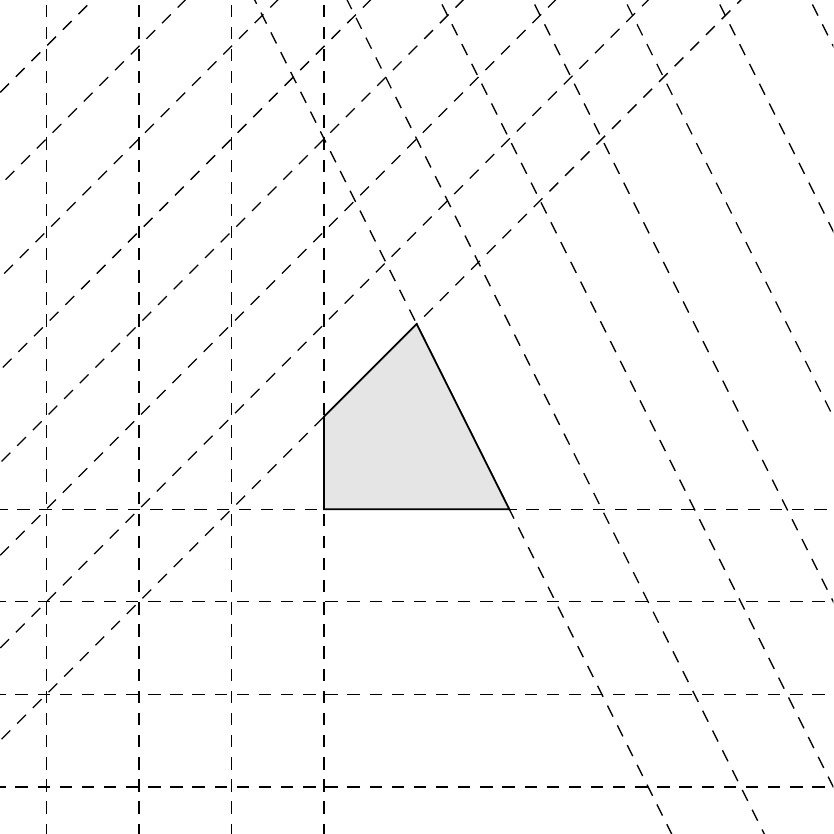}
\vskip.5cm
\caption{ Polar hyperplanes of the Mellin transform $M_{1/f}$. }
\end{figure}

\noindent{\bf Remark.}\ For the Mellin transform of a general rational function $g/f$ each monomial in the numerator $g$ produces an integral similar to the one in the theorem, except that we get a shift in the variable $s$ by an integer vector. This corresponds to a translation of the Newton polytope of $f$, and hence also of the domain of convergence of that particular integral. If $g$ has several monomials it can very well happen that the intersection of all the corresponding shifted polytopes is empty. In that case the integral 
defining the Mellin tranform may not actually converge for any values of $s$. Nevertheless, performing the meromorphic continuation of each of the integrals associated with the monomials from $g$ and then summing these meromorphic functions, we still obtain a natural interpretation of the Mellin transform $M_{g/f}$ as a meromorphic fucntion in the entire $s$-space.
\smallskip

\begin{proof} We prove that the integral (\ref{mellint}) can be re-written in such a way as to make it have a larger convergence domain, at the expense of having to multiply the integral with reciprocals of linear terms corresponding to the poles of the gamma functions.
In order to achieve this we shall repeatedly ``integrate by parts" in each of the directions given by the vectors $\mu_k$. Each such step consists in first making the corresponding dilation $(z_1, \ldots, z_n) \mapsto(\lambda^{\mu_{k1}}z_1,\ldots, \lambda^{\mu_{kn}}z_n)$ of the coordinates, then differentiating with respect to the dilation parameter 
$\lambda$, and finally setting $\lambda$ equal to $1$. 
\smallskip

Note that, if $\Gamma$ is the facet of $\Delta_f$ with inward normal vector $\mu_k$, the truncated polynomial $f_\Gamma$ has the homogeneity $f_\Gamma(\lambda^{\mu_k}z )=\lambda^{\nu_k}f_\Gamma(z)$. Hence, the scaled polynomial $\lambda^{-\nu_k}f(\lambda^{\mu_k}z)$ has the property that all its monomials with exponents from $\Gamma$ have coefficients that are independent of the parameter $\lambda$. This means that in the differentiated polynomial 
$$g_{e_k}(z)=\frac{d}{d\lambda}\Bigl(\lambda^{-\nu_k}f(\lambda^{\mu_k}z)\Bigr)\bigg|_{\lambda=1}$$
\smallskip

\noindent there are no monomials with exponents from the facet $\Gamma$. Its Newton polytope is therefore strictly smaller than $\Delta_f$, with the integer $\nu_k$ from the original inequality $\langle\mu_k,\sigma\rangle\ge\nu_k$ being replaced by $\nu_k+1$, or possibly by an even larger integer.
\smallskip

Starting from the original integral expression (\ref{mellint}) for the Mellin transform $M_{1/f}$, introducing the parameter $\lambda$, and keeping in mind that $M_{1/f}$ itself is of course independent of $\lambda$, we obtain
\smallskip

$$ 0=\frac{d}{d\lambda}\int_{\mathbb{R}_+^n}\frac{(\lambda^{\mu_k}z)^{s}}{{f(\lambda^{\mu_k}z)}}\frac{dz}{z}=\frac{d}{d\lambda}\Bigl[\lambda^{\langle{\mu_k},s\rangle-\nu_k}\int_{\mathbb{R}_+^n}\frac{z^{s}}{{\lambda^{-\nu_k}f(\lambda^{\mu_k}z)}}\frac{dz}{z}\Bigr]\,,$$
\smallskip

\noindent which upon performing the differentation and setting $\lambda=1$ yields the identity
\smallskip

\begin{equation}\label{ett}
\bigl(\langle\mu_k,s\rangle-\nu_k\bigr)\,M_{1/f}(s)=\int_{\mathbb{R}_+^n}\frac{z^{s}g_{e_k}(z)}{f(z)^2}\frac{dz}{z}\,.
\end{equation}
\smallskip

As we shall iterate this procedure it will be important to keep track of polytopes of different sizes, and to this end we
introduce, for any vector $\gamma\in\mathbb{Z}^n$, the notation

$$\Delta(\gamma)=\bigcap_{k=1}^N\bigl\{\sigma\in\mathbb{R}^n\,;\,\langle\mu_k,\sigma\rangle\geq\gamma_k\bigr\}\,.$$
\smallskip

\noindent In particular, we have $\Delta_f=\Delta(\nu)$. Now let $m\in\mathbb{N}^N$ be a given vector, and perform the integration by parts $m_j$ times in the direction of $\mu_j$, for each $j=1,\ldots,N$. The total number of such integrations will thus be $|m|=m_1+\ldots+m_N$. We claim that this iterative process leads to an expression for the Mellin transform that is of the form  

\begin{equation}\label{int2}M_{1/f}(s)=\frac{1}{\prod_{j=1}^N u_j(s)}\int_{\mathbb{R}^n_+}\frac{z^{s}\,g_m(z)}{f(z)^{1+|m|}}\frac{dz}{z},\end{equation} 
\smallskip

\noindent where $g_m$ is a polynomial whose Newton polytope satisfies
$\Delta_{g_m}\!\subseteq\,\Delta(|m|\nu+m)$ and $u_j(s)=\prod_{\ell=0}^{m_j-1}\bigl(\langle\mu_j,s\rangle-\nu_j+\ell\bigr)$, with the convention $u_j=1$ if $m_j=0$.
\smallskip

The proof of the claim is  by induction. First we check that it holds true in the case $|m|=1$, that is, when $m$ is a standard unit vector $e_k$ with $1$ in the $k$'th entry 
and zeros elsewhere. Indeed, this is precisely the content of formula (\ref{ett}), where we recall that the Newton polytope of $g_{e_k}$ is contained in $\Delta(\nu+e_k)$.
\smallskip

Assume now the claim to be true for some given vector $m$, and let us show that it then holds also for $m'=m+e_k$, where $e_k$ is a unit vector as before.
Introducing again the dilated coordinates $\lambda^{\mu_k} z$, we can re-write the integral in equation (\ref{int2}) as

$$\lambda^{\langle\mu_k,s\rangle-\nu_k+m_k}\int_{\mathbb{R}_+^n}\frac{z^s\,\lambda^{-|m|\nu_k-m_k}g_m(\lambda^{\mu_k} z)}{\lambda^{-(1+|m|)\nu_k}
f(\lambda^{\mu_k} z)^{1+|m|}}\frac{dz}{z}\,.$$
\smallskip

\noindent We should then differentiate this expression with respect to $\lambda$ and put $\lambda=1$. When the derivative falls on the monomial in front of the integral we get a factor $\langle\mu_k,s\rangle-\nu_k+m_k$ which is precisely what needs to be incorporated into the function $u_k$, and when we differentiate under the sign of integration we arrive at an expression of the form
$$-\int_{\mathbb{R}_+^n}\frac{z^s\,g_{m'}(z)}{f(z)^{2+|m|}}\frac{dz}{z}\,.$$
\smallskip

\noindent The new polynomial in the numerator is $g_{m'}(z)=(1+|m|)g_{e_k}(z)g_m(z)-f(z)\tilde g_m(z)$, where
$$\tilde g_m(z)=\frac{d}{d\lambda}\Bigl(\lambda^{-|m|\nu_k-m_k}g_m(\lambda^{\mu_k} z)\Bigr)\bigg|_{\lambda=1}\,.$$
\smallskip

\noindent To finish the proof of the claim we must show that $\Delta_{g_{m'}}\!\subseteq\,\Delta(|m'|\nu+m')$. We shall use the fact that
the Newton polytope of a product of two polynomials is equal to the (Minkowski) sum of their Newton polytopes, and also the obvious general inclusion
$\Delta(\gamma)+\Delta(\delta)\subseteq\Delta(\gamma+\delta)$. Recalling the induction hypothesis, we first see that the Newton polytope of the product 
$g_{e_k}g_m$ is contained in the polytope
$\Delta(\nu+e_k)+\Delta(|m|\nu+m)\subseteq\Delta((1+|m|)\nu+m+e_k)=\Delta(|m'|\nu+m')$. Then, since the polynomial $\tilde g_m$ has no monomials
with exponents on the plane $\langle\mu_k,\sigma\rangle=|m|\nu_k+m_k$, we similarly get that the Newton polytope of the other term $f\tilde g_m$ is contained in
$\Delta(\nu)+\Delta(|m|\nu+m+e_k)\subseteq\Delta(|m'|\nu+m')$. From this the claim follows, that is, the Mellin transform is given by (\ref{int2}) with $g_m$ 
satisfying $\Delta_{g_m}\!\subseteq\,\Delta(|m|\nu+m)$.
\smallskip

Our next step is to prove that the integral in (\ref{int2}) converges and defines an analytic function for all $s$ with real parts $\sigma$ in the enlarged polytope $\Delta(\nu-m)$.
By considering separately each term of $g_m$, we can infer from Theorem~1 that the domain of convergence will contain (the interior of) the intersection
\begin{equation}\label{snitt}\bigcap_{\tau\in \Delta_{g_m}}[(1+|m|)\Delta_f-\tau]\end{equation}
of translates of dilated copies of $\Delta_f$.
\smallskip

Let us check that $\Delta(\nu-m)$ is indeed a subset of (\ref{snitt}). Take an arbitrary $\sigma_0\in\Delta(\nu-m)$. By definition it satisfies the inequalities
\begin{equation}\label{olikhet}\langle\mu_k,\sigma_0\rangle\ge\nu_k-m_k\,,\qquad   k=1,\ldots,N\,.\end{equation}
In order to see that $\sigma_0$ also belongs to the intersection (\ref{snitt}), take any $\tau\in\Delta_{g_m}$ and observe that the polytope $(1+|m|)\Delta_f-\tau$ is given by the inequalities
\begin{equation}\label{olik1}\langle\mu_k,\sigma+\tau\rangle\ge(1+|m|)\nu_k\,,\qquad   k=1,\ldots,N\,.\end{equation}
What we have to show is that $\sigma_0$ satisfies these inequalities. In view of the inclusion $\Delta_{g_m}\!\subseteq\,\Delta(|m|\nu+m)$, we have
$\langle\mu_k,\tau\rangle\ge|m|\nu_k+m_k$ for all $k$. Together with (\ref{olikhet}) this gives
$$\langle\mu_k,\sigma_0+\tau\rangle=\langle\mu_k,\sigma_0\rangle+\langle\mu_k,\tau\rangle\ge\nu_k-m_k+|m|\nu_k+m_k=(1+|m|)\nu_k\,,$$
so $\sigma_0$ does indeed satisfy (\ref{olik1}), and since $\tau$ was arbitrary it follows that $\sigma_0$ lies in the intersection (\ref{snitt}).
\smallskip

In the interior of the domain $\Delta(\nu-m)+i\,\mathbb{R}^n$ the only poles of $M_{1/f}$ are given by $u_j(s)=0$, $j=1,\ldots,N$. All these poles are simple. This is the same polar locus as for the product $\prod_k\Gamma(\langle\mu_k,s\rangle-\nu_k)$. By the theorem on removable singularities it follows that the quotient $M_{1/f}/\prod_k\Gamma(\langle\mu_k,s\rangle-\nu_k)=\Phi$ is holomorphic for $\sigma$ inside the polytope $\Delta(\nu-m)$. But here $m\in\mathbb{N}^N$ is arbitrary, and since the union of all the $\Delta(\nu-m)$ is the entire space $\mathbb{R}^n$, we conclude that $\Phi$ is in fact an entire function as claimed in the theorem. \end{proof}
\bigskip

\section{Two special cases}

\noindent In certain situations we are able to make our description of the Mellin transform even more precise, and explicitly compute the entire function $\Phi$ that occurs in front of the gamma factors in Theorem~2. We have already encountered such a case in Example~1 of the introduction, where we considered the transform of the simple fraction $1/(1+z_1+z_2)$. Elaborating this example just a little further, and considering a more general linear fraction $1/(c_0+c_1z_1+\ldots+c_nz_n)$ with each coefficient $c_k$ being a positive real number, one easily deduces the formula
\begin{equation}\label{fraction}M_{1/f}(s)=c_0^{s_1+\ldots+s_n-1}c_1^{-s_1}\cdots c_n^{-s_n}\,\Gamma(s_1)\cdots\Gamma(s_n)\Gamma(1-s_1-\ldots-s_n)\,.\end{equation}
So in this case the entire function $\Phi$ is equal to the elementary exponential function $s\mapsto c_0^{s_1+\ldots+s_n-1}c_1^{-s_1}\cdots c_n^{-s_n}$ and in particular different from zero everywhere.
\smallskip

We shall now consider two families of examples that both generalize the case of a linear fraction, namely products of linear fractions and rational functions that are obtained from linear fractions by means of a monomial change of variables.
\smallskip

\begin{pro}\label{mellin} Assume that the polynomial $f(z)=\prod_{k=0}^m\bigl(1+\langle a_k,z\rangle\bigr)$ is a product of affine linear 
factors, with each $a_k\in\mathbb{R}^n_+.$ Then the Mellin transform of the rational function $1/f$ is equal to
\be\label{f} M_{1/f}(s)=\Phi(s)\Gamma(s_1)\ldots\Gamma(s_n)\Gamma(m+1-s_1-\ldots -s_n)\,,\ee
with the entire function $\Phi$ given by
$$\Phi(s)=\int_{\sigma_m}\frac{d\tau_1\cdots d\tau_m}{\alpha_1(\tau)^{s_1}\cdots\alpha_n(\tau)^{s_n}}\,.$$ 
Here $\sigma_m$ denotes the standard $m$-simplex $\bigl\{\,\tau\in\mathbb{R}^m_+\,;\ \sum\tau_k<1\,\bigr\}$, and 
the $\alpha_k(\tau)$ are affine linear forms defined by 
$$\bigl(\alpha_1(\tau),\ldots,\alpha_n(\tau)\bigr)=\bigl(1-\text{$\sum$}\tau_k\bigr)a_0+\tau_1a_1+\ldots+\tau_m a_m\,.$$
 
\end{pro}

\begin{proof} We begin by first computing the Mellin transform of a power of the type $1/(1+\langle c,z\rangle)^{m+1}$.
By performing repeated integrations under the sign of integration we get
$$M_{1/(1+\langle c,z\rangle)^{m+1}}=\frac{(-1)^m}{m!}\frac{d^m}{d\lambda^m}M_{1/(\lambda+\langle c,z\rangle)}
\Big|_{\lambda=1}\,.$$

\noindent Then, recalling the formula (\ref{fraction}) and using the simple identity
$$\frac{(-1)^m}{m!}\frac{d^m}{d\lambda^m}\lambda^{s_1+\ldots+s_n-1}\Big|_{\lambda=1}=\frac{1}{m!}
\frac{\Gamma(m+1-s_1-\ldots-s_n)}{\Gamma(1-s_1-\ldots-s_n)}\,,$$
we find that 
$$M_{1/(1+\langle c,z\rangle)^{m+1}}=\frac{1}{m!}c_1^{-s_1}\cdots c_n^{-s_n}\Gamma(s_1)\cdots\Gamma(s_n)\Gamma(m+1-s_1-\ldots-s_n)\,.$$

\noindent Next we make use of the generalized partial fractions decomposition

$$\frac{1}{\prod_{k=0}^m(1+\langle a_k,z\rangle)}=m!\int_{\sigma_m}\frac{d\tau_1\cdots d\tau_m}{(1+\langle \alpha(\tau),z\rangle)^{m+1}}\,,$$
\smallskip

\noindent which occurs in the theory of analytic functionals and Fantappi\`e transforms, see for instance \cite{APS} or \cite{H3}. 
From this formula we immediately obtain 
$$ M_{1/\prod_{k=0}^m(1+\langle a_k,z\rangle)}=m!\int_{\sigma_m}M_{1/(1+\langle\alpha(\tau),z\rangle)^{m+1}} 
\,d\tau_1\cdots d\tau_m\,,$$

\noindent which yields (\ref{f}). \end{proof}
\bigskip

In particular, when $m=n=1$ and $f(z)=(1+a_0z)(1+a_1 z)$ we obtain the entire function
$$\Phi(s)=\int_0^1\frac{d\tau}{\bigl((1-\tau)a_0+\tau a_1\bigr)^s}\frac{1}{1-s}\frac{\ a_1^{1-s}\!-\,a_0^{1-s}\!\!\!}{a_1-\,a_0}\,=\frac{1}{1-s}\sum \mathrm{res}[z^{s-1}/f(z)]$$
in accordance with the formulas mentioned in the introduction above. Similarly, when $n=2$ and 
$f(z_1,z_2)=(1+a_{01}z_1+a_{02}z_2)(1+a_{11}z_1+a_{12}z_2)$ the entire function becomes
$$\Phi(s_1,s_2)=\int_0^1\frac{d\tau}{\bigl((1-\tau)a_{01}+\tau a_{11}\bigr)^{s_1}\bigl((1-\tau)a_{02}+\tau a_{12}\bigr)^{s_2}}\,.$$
Here one may remark a close connection to the classical Euler beta function $B$. Namely, if we let the coefficients $a_{02}$
and $a_{11}$ become zero, we are left with
$$\Phi(s_1,s_2)=a_{01}^{-s_1}a_{12}^{-s_2}\int_0^1(1-\tau)^{-s_1}\tau^{-s_2}d\tau=a_{01}^{-s_1}a_{12}^{-s_2}B(1-s_1,1-s_2)\,.$$
Since $B(1-s_1,1-s_2)=\Gamma(1-s_1)\Gamma(1-s_2)/\Gamma(2-s_1-s_2)$ we see that the function $\Phi$ is no longer entire. 
This is to be expected however, because the new polynomial $f(z_1,z_2)=(1+a_{01}z_1)(1+a_{12}z_2)$ has a different Newton polygon,
and the new $\Phi$ should contribute to the change of $\Gamma$-factors in the Mellin transform. In fact, when $a_{02}=a_{11}=0$
we have the formula 
$$M_{1/f}(s)=\Phi(s)\Gamma(s_1)\Gamma(s_2)\Gamma(2-s_1-s_2)=a_{01}^{-s_1}a_{12}^{-s_2}\Gamma(s_1)\Gamma(s_2)\Gamma(1-s_1)
\Gamma(1-s_2)\,.$$
\smallskip

It is not always the case that all the polar hyperplanes of the gamma functions in the representation (\ref{mero}) are actual singularities for the
Mellin transform $M_{1/f}$. It may happen that the entire function $\Phi$ has zeros that cancel out some of the poles. A very simple
example of this phenomenon is provided by the function $f(z)=1+z^m$ with $m\ge2$. In this case the substitution $z^m=w$ leads to the formula
$$M_{1/f}(s)=\frac{1}{m}\int_0^\infty\frac{w^{s/m}}{1+w}\frac{dw}{w}=\frac{1}{m\,}\Gamma(s/m)\Gamma(1-s/m)\,,$$
so the polar locus is just $m\mathbb{Z}$. In fact, the entire function $\Phi$ from (\ref{mero}) is given by
$$\Phi(s)=\frac{1}{m}\,\frac{\Gamma(s/m)}{\Gamma(s)}\frac{\Gamma(1-s/m)}{\Gamma(m-s)}\,,$$
and it has plenty of integer zeros. A slight generalization of this example is provided by the following result.
\smallskip

\begin{pro}\label{monompro} Let $f(z)=1+z^{\alpha_1}+\ldots+z^{\alpha_n}$, for some linearly independent vectors $\alpha_1,\ldots,\alpha_n\in\mathbb{Z}^n$, 
and denote by $\delta$ the non-zero determinant $\det(\alpha_{jk})$. The Mellin transform of the rational function $1/f$ is then given by
$$M_{1/f}(s)=\frac{1}{\delta}\,\Gamma(\langle\beta_1,s\rangle)\cdots\Gamma(\langle\beta_n,s\rangle)\Gamma(1-\langle\beta_1,s\rangle-\ldots-\langle\beta_n,s\rangle)\,,$$ 
where the $\beta_k$ denote the column vectors of the inverse matrix $(\alpha_{jk})^{-1}$.
\end{pro}

\begin{proof} We make the monomial change of variables $z_1^{\alpha_{j1}}\cdots z_n^{\alpha_{jn}}=w_j$, so that $z_j=w_1^{\beta_{j1}}\cdots w_n^{\beta_{jn}}$ and
$dz/z=\delta^{-1}dw/w$. The Mellin transform can then be written
$$M_{1/f}(s)=\int_{\mathbb{R}^n_+}\frac{z^{s}}{1+z^{\alpha_1}+\ldots+z^{\alpha_n}}\frac{dz}{z}=\frac{1}{\delta}\int_{\mathbb{R}^n_+}\frac{w_1^{\langle\beta_1,s\rangle}\!\cdots \,w_n^{\langle\beta_n,s\rangle}}{1+w_1+\ldots+w_n}\frac{dw}{w}\,,$$ 
and the latter integral is of a similar form as the one in Example 1.\end{proof}
\bigskip

We point out that the Newton polytope $\Delta_f$ of the polynomial in Proposition~\ref{monompro} is a simplex with one vertex at the origin, and that its normal vectors
$\mu_1,\ldots,\mu_{n+1}$ are integer multiples of the rational vectors $\beta_1,\ldots,\beta_n$ and $-(\beta_1+\ldots+\beta_n)$. Moreover, one has $\nu_1=\ldots=\nu_n=0$ and $\nu_{n+1}=1$. In this case the entire function $\Phi$ 
occurring in (\ref{mero}) is therefore of the form
$$\Phi(s)=\frac{1}{\delta}\,\frac{\Gamma(\langle\beta_1,s\rangle)}{\Gamma(\langle\mu_1,s\rangle)}\cdots
\frac{\Gamma(\langle\beta_n,s\rangle)}{\Gamma(\langle\mu_n,s\rangle)}\frac{\Gamma(1-\langle\beta_1,s\rangle-\ldots-\langle\beta_n,s\rangle)}{\Gamma(1+\langle\mu_{n+1},s\rangle)}\,.$$
\bigskip

\section{Mellin transforms and coamoebas}

\noindent Let us return for a moment to the one-variable Mellin transform 

$$M_{1/f}(s)=\int_0^\infty\frac{z^s}{f(z)}\frac{dz}{z}=\int_{-\infty}^\infty\frac{e^{sx}}{f(e^x)}dx\,,$$
\smallskip

\noindent where we assume, as before, that the polynomial $f$ does not vanish on the positive real axis and that the real part of $s$ lies in the interior of the Newton interval $\Delta_f$. Our first claim is now that the value of the above integral remains unchanged if the set of integration is rotated slightly. In other words, for $|\theta|$ small enough one has the identity
$$\int_0^\infty\frac{z^s}{f(z)}\frac{dz}{z}=\int_{\text{Arg}^{-1}(\theta)}\frac{z^s}{f(z)}\frac{dz}{z}=\int_{-\infty}^\infty\frac{e^{s(x+i\theta)}}{f(e^{x+i\theta})}dx\,.$$
\smallskip

\noindent To verify this, we perform an integration along a closed path starting at the origin, then running along the positive real axis to the point $R$, continuing along the circle $|z|=R$ to the point $Re^{i\theta}$, and then going straight back to the origin, see Figure~4 below. Since $\theta$ is close to zero, the denominator $f$ has no zeros in the closed sector with arguments between $0$ and $\theta$. By the residue theorem the integral over the closed contour is therefore equal to zero, and since the integrand decreases fast when $|z|\to\infty$, the integral over the circular arc $C_R$ can be made arbitrarily small by choosing $R$ large enough. The integrals along the two infinite rays are thus equal as claimed.

\vskip-1cm

\vskip-4cm
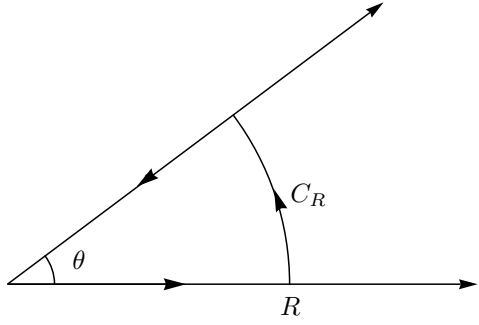
\begin{figure}[H]
\hskip2.5cm
\includegraphics[height=10cm]{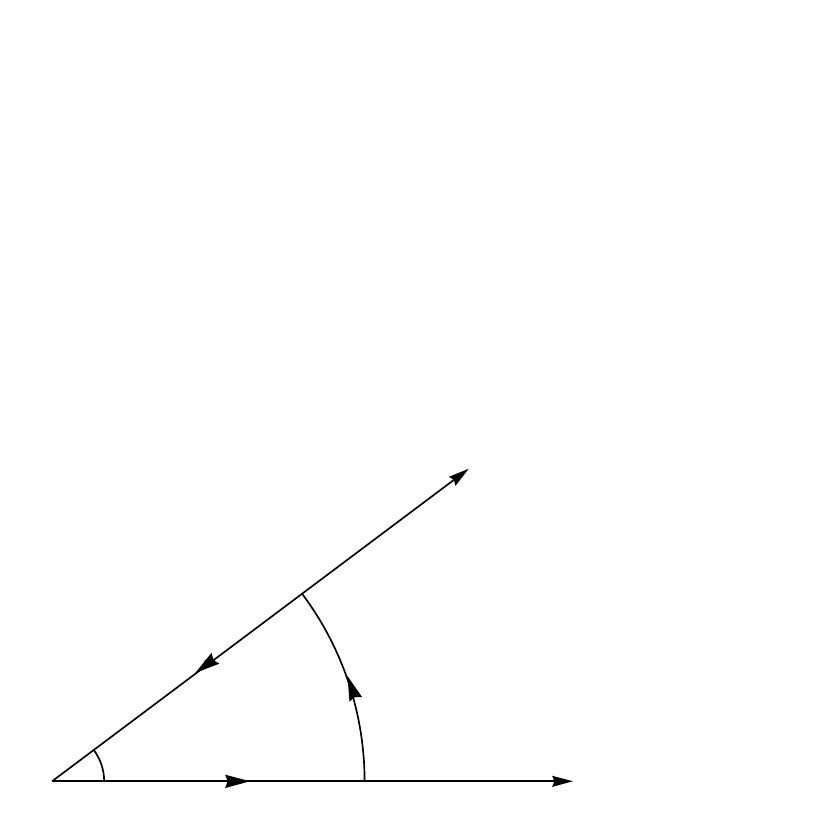}
\begin{picture}(0,0)
\put(-245.5,23.5){$\theta$}
\put(-163,49){$C_R$}
\put(-167,6){$R$}

\end{picture}

\vskip.5cm
\caption{The contour of integration in the residue computation.}
\end{figure}

From the above argument we see that the directional Mellin transform coincides with the standard one as long as the two directions $\theta$ and $0$ belong to the same connected component of the coamoeba complement $\mathbb{R}\setminus\mathcal{A}'_f$. Furthermore, it is clear that the Mellin integral over $\text{Arg}^{-1}(\theta)$ converges for \emph{every} choice of $\theta$ outside the coamoeba $\mathcal{A}'_f$. A similar residue computation as above then again shows that the directional Mellin transform only depends on which connected component of $\mathbb{R}\setminus\mathcal{A}'_f$ it is that contains $\theta$.
\smallskip

Turning to the general case $n\ge1$, there are two important differences to be observed. On the one hand we recall the condition in Theorem~\ref{expgro} that the polynomial $f$ should be completely non-vanishing on $\mathbb{R}^n_+=\text{Arg}^{-1}(0)$ in order for the integral to converge, and on the other hand we note that the coamoeba 
$\mathcal{A}'_f$ is in general not a closed set. The following result connects these two facts, and it allows us to define the directional Mellin transform 
\be\label{mellint2} \int_{\Arg^{-1}(\theta)} \frac{z^{s}}{f(z)}\frac{dz}{z}=\int_{\mathbb{R}^n} \frac{e^{\langle s,x+i\theta\rangle}}{f(e^{x+i\theta})}dx\,,\ee
for each argument $\theta$ that does not belong to the closure $\overline{\mathcal{A}'_f}$.
\smallskip

\begin{thm}\label{theta}
For any $\theta\in\mathbb{R}^n\setminus\overline{\mathcal{A}'_f}$ the polynomial $f$ is completely non-vanishing on the set ${\rm Arg}^{-1}(\theta)$.
\end{thm}
\smallskip

\begin{proof}
For any given argument vector $\theta$ we can consider the new polynomial $f^\theta(z)=f(e^{i\theta_1}z_1,\ldots, e^{i\theta_n}z_n)$. 
Observe that $\mathcal{A}'_{f^\theta}+\theta=\mathcal{A}'_f$, so that $0\in\mathcal{A}'_{f^\theta}$ if and only if $\theta\in\mathcal{A}'_{f}$, and also that $f^\theta$ is completely non-vanishing on ${\rm Arg}^{-1}(0)$ if and only if $f$ is completely non-vanishing on ${\rm Arg}^{-1}(\theta)$. This means that it actually suffices to prove the theorem for the special case $\theta = 0$.
\smallskip

Assume then that $f$ is not completely non-vanishing on the set ${\rm Arg}^{-1}(0)$, so that for some face $\Gamma$ one has $0\in\mathcal{A}'_\Gamma$. In other words,
there is an $x_0\in\mathbb{R}^n$ such that $f_\Gamma(e^{x_0})=0$. We must show that $0$ belongs to the closure $\overline{\mathcal{A}'_f}$. This is obvious if $\Gamma=\Delta_f$, so we can assume that $\dim\Gamma\le n-1$. Choose a vector $\mu\in\mathbb{Z}^n$ and an integer $\nu\in\mathbb{Z}$ such that $\langle\mu,\alpha\rangle=\nu$ for $\alpha\in\Gamma$ and $\langle\mu,\alpha\rangle<\nu$ for $\alpha\in\Delta_f\setminus\Gamma$.
Writing $g_\Gamma=f-f_\Gamma$ we then have
\begin{eqnarray*}f_\Gamma(e^{x_0-t\mu})=\sum_{\alpha\in\Gamma}a_\alpha \,e^{\langle x_0,\alpha\rangle-t\langle\mu,\alpha\rangle}=e^{-t\nu}\sum_{\alpha\in\Gamma}a_\alpha \,e^{\langle x_0,\alpha\rangle}=0\end{eqnarray*} 
\noindent and
\begin{eqnarray*} g_\Gamma(e^{x_0-t\mu})=\sum_{\alpha\in\Delta_f\setminus\Gamma}a_\alpha \,e^{\langle x_0,\alpha\rangle-t\langle\mu,\alpha\rangle}=e^{-t\nu}\!\!\!\!\sum_{\alpha\in\Delta_f\setminus\Gamma} b_\alpha\,e^{-tc_\alpha}\,,\end{eqnarray*}
where $b_\alpha=a_\alpha e^{\langle x_0,\alpha\rangle}$ and $c_\alpha=\langle\mu,\alpha\rangle-\nu >0$.
Now let $\varepsilon>0$ be given. Choose a disk $D_\varepsilon$ of radius $\varepsilon$ centered at $x_0$ and contained in a complex line on which the function 
$w\mapsto f_\Gamma(e^w)$ does not vanish identically. Then translate this disk along the real space, so that 
$D_\varepsilon-t\mu$ is a disk centered at the point $x_0-t\mu$ for some large positive number $t$. Since $f_\Gamma(e^w)$ is non-zero on the boundary of 
$D_\varepsilon$ we have $|f_\Gamma(e^w)|\ge\delta>0$ for $w\in\partial D_\varepsilon$. This means that $|f_\Gamma(e^w)|\ge\delta\,e^{-t\nu}$ on the translated circle $\partial D_\varepsilon-t\mu$. Taking $t$ large enough, we also have $|g_\Gamma(e^w)/f_\Gamma(e^w)|<1$ on $\partial D_\varepsilon-t\mu$, that is, $|g_\Gamma(e^w)|<|f_\Gamma(e^w)|$.
Rouch\'e's theorem then tells us that $f(e^w)=f_\Gamma(e^w)+g_\Gamma(e^w)$ has a zero $w_\varepsilon$ in the disk $D_\varepsilon-t\mu$. So $z_\varepsilon=e^{w_\varepsilon}$ belongs to the hypersurface $f(z)=0$. But we also know that $|\!\Arg(z_\varepsilon)|=|\mathrm{Im}\,w_\varepsilon|<\varepsilon$, and since
 $\varepsilon $ was chosen arbitrary we conclude that $0\in\overline{\mathcal{A}'_f}$.\end{proof}
\smallskip

\noindent{\bf Remark.}\ In the above proof we showed that all the facial coamoebas $\mathcal{A}'_\Gamma$ are contained in the closure $\overline{\mathcal{A}'_f}$ of the main coamoeba. It is a fact, proved by Johansson \cite{J} and independently by Nisse and Sottile \cite{NS}, that one actually has an equality
$$\bigcup_{\Gamma\subseteq\Delta_f}\mathcal{A}'_\Gamma=\overline{\mathcal{A}'_f}\,.$$
\smallskip

Using Theorems~\ref{expgro} and \ref{theta} we can now define a directional Mellin transform (\ref{mellint2}) for any $\theta$ in the complement $\mathbb{R}^n\setminus\overline{\mathcal{A}'_f}$. Just as in the one-variable case discussed earlier in the section, the various Mellin transforms will in fact be equal for all $\theta$
that belong to the same connected component of  $\mathbb{R}^n\setminus\overline{\mathcal{A}'_f}$. This can be seen by connecting two different values of $\theta$ through a polygonal path such that along each edge of the path only one component $\theta_k$ is being changed. The invariance of the Mellin transform under such a move is then a consequence of the one-variable argument.
\smallskip

ln order to  put our next theorem in a proper perspective, it seems appropriate at this juncture to recall some known facts about amoebas and Laurents series of rational functions. 
A reference for these results is \cite{FPT}. Associated with each connected component $E$ of the amoeba complement $\mathbb{R}^n\setminus\mathcal{A}_f$ is a Laurent series
representation 
$$\frac{1}{f(z)}=\sum_{\alpha\in\mathbb{Z}^n}c^E_\alpha \,z^{-\alpha}$$
of the rational function $1/f$. The coefficients of the series are given by the integrals

$$c^E_\alpha=\frac{1}{(2\pi i)^n}\int_{\Log^{-1}(x)}\frac{z^\alpha}{f(z)}\frac{dz}{z}=\int_{[-\pi,\pi]^n}\frac{e^{\langle\alpha,x+i\theta\rangle}}{f(e^{x+i\theta})}\,d\theta\,,$$
\smallskip

\noindent  where $x$ is any point in the connected component $E$. Each such Laurent series will converge in the corresponding Reinhardt domain $\Log^{-1}(E)$. We stress the fact that the amoeba $\mathcal{A}_f$ is always a closed set, so in contrast to the case of coamoebas, there is no need to take the closure of the amoeba.
\smallskip 

The following result about coamoebas and Mellin transforms provides a practically perfect analogy to the above picture for amoebas and Laurent coefficients.
 
\begin{thm}\label{inversmellin}
For any connected component $E$ of the coamoeba complement $\mathbb{R}^n\setminus\overline{\mathcal{A}'_f}$ there is an integral representation 
\begin{equation}\label{invers} \frac{1}{f(z)}=\int_{\sigma+i\mathbb{R}^n}M^E_{1/f}(s)\,z^{-s}ds\,,\end{equation}
which converges for all $z$ in the domain $\text{\rm Arg}^{-1}(E)$. Here $\sigma$ is an arbitrary point in $\mathrm{int}\,\Delta_f$ and 
\begin{equation}\label{component} M^E_{1/f}(s)=\frac{1}{(2\pi i)^n}\int_{\Arg^{-1}(\theta)}\frac{z^{s}}{f(z)}\frac{dz}{z}=\frac{1}{(2\pi i)^n}\int_{\mathbb{R}^n}\frac{e^{\langle s,x+i\theta\rangle}}{f(e^{x+i\theta})}dx\,,\end{equation}
with $\theta$ being an arbitrary point in the component $E$.
\end{thm}
\smallskip

\begin{proof}
From Theorem~\ref{theta} and (an obvious generalization of) Theorem~\ref{expgro} we see that the integral (\ref{component}) converges, and from the discussion preceding Theorem~\ref{theta} we also know that the value of (\ref{component}) is independent of the particular choice of point $\theta\in E$.

In order to prove the identity (\ref{invers}) it suffices to verify that, for all $s=\sigma+it$ such that $\sigma\in\mathrm{int}(\Delta_f)$, the function 
$x\to e^{\langle s,x+i\theta\rangle}/f(e^{x+i\theta})$ is in the Schwartz space $\mathcal{S}(\mathbb{R}^n)$ of rapidly decreasing functions. Then the result follows from well known facts about inversion of Fourier transforms, see Thm $7.1.5$ in \cite{H2}.
\smallskip

For simplicity, and without loss of generality, we assume that $\theta=0$. We have $|e^{\langle s,x\rangle}/f(e^{x})|=e^{\langle \sigma,x\rangle}/|f(e^{x})|$, and from the inequality (\ref{olik}), which we established in the proof of Theorem~\ref{expgro}, we see that $e^{\langle s,x\rangle}/f(e^{x})$ is an exponentially dercreasing function. 
It remains to verify that all its partial derivatives have the same property. Computing a typical derivative, we get
\begin{equation}\label{deriv}\frac{\partial}{\partial x_k}\Bigl(\frac{e^{\langle\sigma,x\rangle}}{f(e^x)}\Bigr)=\frac{\,\sigma_k\,e^{\langle\sigma,x\rangle}}{f(e^x)}-\frac{e^{\langle\sigma+e_k,x\rangle}f'_k(e^x)}{f(e^x)^2}\,,\end{equation}
where $f'_k$ denotes the derivative of the polynomial $f$ with respect to $z_k$. Here the first term one the right hand side is just a constant times the original function, and the second term is of the form
$$\sum_{\alpha\in A}\frac{\alpha_ka_\alpha\, e^{\langle\sigma+\alpha,x\rangle}}{f(e^x)^2}\,.$$
The Newton polytope of the denominator is $\Delta_{f^2}=2\Delta_f$, so $\sigma+\alpha\in\text{int}\,\Delta_{f^2}$ for every $\alpha\in A$, and hence each term in the sum satisfies the conditions of Theorem~ \ref{expgro}. This means that the derivative (\ref{deriv}) is a finite sum of functions to which we can apply Theorem~\ref{expgro} and the inequality 
(\ref{olik}). By induction this implies that all derivatives of $e^{\langle s,x\rangle}/f(e^{x})$ decrease exponentially. 
\end{proof}

\noindent\textbf{Remark}. It is clear that if $E'$ is a connected component of $\mathbb{R}^n\setminus\overline{\mathcal{A}'_f}$ that is obtained by just translating
another component $E$  by $2\pi e_k$, then the corresponding two Mellin transforms are related by the simple formula
$$M^{E'}_{1/f}(s)=e^{2\pi i s_k}M^E_{1/f}(s)\,.$$
In general however, the relations between the various Mellin transforms associated with different connected components are rather complicated. Furthermore, it is worth mentioning that all the connected components of $\mathbb{R}^n\setminus\overline{\mathcal{A}'_f}$ are convex sets. This fact follows for instance from the Bochner tube theorem, see \cite{Bo}.
Finally, we point out that Theorem~\ref{inversmellin} can also be proved by using results from Antipova \cite{A}.
\bigskip

\section{Hypergeometry}

\noindent In this final section we shall consider the dependence of the Mellin transform, and in particular of the entire function $\Phi$, on the coefficients $a=\{a_\alpha\}$ of the polynomial $f$. In order to emphasize this dependence we are here going to write $\Phi(a,s)$ rather than just $\Phi(s)$. The crucial observation will be that, with respect to the variables $a$,  the function $\Phi$ is an $A$-hypergeometric function in the sense of Gelfand, Kapranov and Zelevinsky. More precisely, $a\mapsto\Phi(a,s)$ satisfies the
 $A$-hypergeometric system of partial differential equations with homogeneity parameter $\beta=(-1,-s_1,-s_2,\ldots,-s_n).$ 
\smallskip

Let us recall the structure of the $A$-hypergeometric system. Our starting point is the subset $A\subset\mathbb{Z}^n$ of exponent vectors occurring in the expression (\ref{polyn}) for the polynomial $f$. We introduce a numbering $\alpha_1,\ldots,\alpha_N$ of the elements of $A$, with each $\alpha_k=(\alpha_{1k},\ldots,\alpha_{nk})\in\mathbb{Z}^n$. Abusing the notation slightly, we write $A$ also for the $(1+n)\times N$-matrix whose column vectors are $(1,\alpha_k)$. For any vector $v\in\mathbb{Z}^n$ we denote by $v^+$ and $v^-$ the vectors obtained from $v$ by replacing each component $v_k$ by $\max(v_k,0)$ and $\max(-v_k,0)$ respectively, so that $v=v^+\!-v^-$.
\smallskip

\begin{dfn} Let $A$ denote a subset $\{\alpha_1,\ldots,\alpha_N\}\subset\mathbb{Z}^n$ and the associated $(1+n)\times N$-matrix as above. The $A$-\emph{hypergeometric system} of differential equations with homogeneity parameter $\beta\in\mathbb{C}^n$ is then given by
$$\square_b F(a)=0, \quad b\in\mathbb{Z}^N,\ Ab=0, \qquad\text{and}\qquad  E_j^\beta F(a) = 0, \quad j=0,1,\ldots,n\,,$$
where the differential operators $\square_b$ and $E^\beta_j$ are given by
$$\square_b=\Bigl(\frac{\partial}{\partial a}\Bigr)^{\!b^+}\!\!\!-\,\Bigl(\frac{\partial}{\partial a}\Bigr)^{\!b^-}\quad\text{and}\quad E^\beta_j=\sum_{k=1}^N\alpha_{jk}\,a_k\frac{\partial}{\partial a_k}\ -\ \beta_j\,.$$
An analytic function $F$ that solves the system is called $A$-hypergeometric with homogeneity parameter $\beta$.
\end{dfn}\smallskip

\noindent{\bf Remark.}\ We are assuming $N\ge1+n$, and as soon as this inequality is strict there are of course infinitely many vectors $b$ satisfying $Ab=0$, but it is a known fact, see \cite{SST}, that the system is in fact determined by a finite number of operators $\square_b$.
\smallskip

Let us now, for a given choice of coefficients $a$, consider an entire function $s\mapsto\Phi(a,s)$ as described in Theorems~\ref{mel} and \ref{inversmellin}. We want to study what happens when we start varying $a$. Recall from \cite{GKZ} and \cite{GKZa} the notion of the principal $A$-determinant $E_A$, also known as the full $A$-discriminant. It is a polynomial in the variables $a$, with the property that its zero set $\Sigma_A\subset\mathbb{C}^N$ contains the singular locus of all $A$-hypergeometric functions.
\smallskip

\begin{thm}\label{hyper}
Take $a\in\mathbb{C}^N\setminus\Sigma_A$ and let $E$ be a connected component of $\mathbb{R}^n\setminus\overline{\mathcal{A}'_f}$, with $f$ being the polynomial
$f(z)=a_1z^{\alpha_1}+\ldots+a_Nz^{\alpha_N}$. Also take $s\in\mathbb{C}^n$ with $\text{\rm Re}\,s\in\text{\rm int}\,\Delta_f$. Then the analytic germ
$$\Phi(a,s)=\frac{1}{\prod_k\Gamma(\langle\mu_k,s\rangle-\nu_k)}\int_{\text{\rm Arg}^{-1}(\theta)}\frac{z^s}{\sum_k a_k z^{\alpha_k}}\frac{dz}{z}\,,\qquad\theta\in E\,, $$
has a (multivalued) analytic continuation to $(\mathbb{C}^N\setminus\Sigma_A)\times\,\mathbb{C}^n$ which is everywhere $A$-hypergeometric in $a$ with varying homogeneity parameter $(-1,-s_1,\ldots,-s_n)$.
\end{thm}

\begin{proof} First of all it is clear that $\theta$ will be disjoint from $\overline{\mathcal{A}'_f}$ also for poly\-nomials $f$ with coefficients $a_k$ near the original ones, say in a small ball $B(a)$, so that the integral does indeed define an analytic germ $\Phi(a,s)$. From (a straightforward generalization of) our Theorem~\ref{mel} we also know that $\Phi$ is extendable as an entire function with respect to the variables $s$. In other words, we already have an analytic extension of $\Phi$ to the infinite cylinder $B(a)\times \mathbb{C}^n$. 
\smallskip

Let us next verify that $\Phi$ is an $A$-hypergeometric function with the correct homogeneity parameter. When doing this we first fix $s$ at an arbitrary value with $\text{Re}\,s\in\text{int}\,\Delta_f$, hence in particular away from the polar hyperplanes of the gamma functions. Then the function in front of the integral is just a non-zero constant and we can deal directly with the integral, by differentiation under the integral sign. 
\smallskip

Notice that the condition that $Ab=0$ amounts to the two identities $|b^+|=|b^-|$ and $\langle b^+,\alpha\rangle=\langle b^-,\alpha\rangle$, were we have used the shorthand notation $|b^\pm|=\sum b^\pm_k$ and $\langle b^\pm,\alpha\rangle=\sum b^\pm_k\alpha_k$. Computing iterated derivatives of the integrand $1/f$ in the Mellin integral we get

$$\Bigl(\frac{\partial}{\partial a}\Bigr)^{b^\pm}\!\frac{1}{\sum a_k\,z^{\alpha_k}}\ =\ (-1)^{|b^\pm|}|b^\pm|!\,\frac{z^{\langle b^\pm,\,\alpha\rangle}}{(\sum a_k\,z^{\alpha_k})^{1+|b^\pm|}}\,, $$
\smallskip

\noindent and since here the right hand side is independent of the choice of sign in $b^\pm$, so is the left hand side. This means that $\square_b (1/f)=0$, and hence we also have $\square_b \Phi=0$.
\smallskip

It is obvious that $\Phi$ is homogeneous of degree $-1$ with respect to the variables $a_k$. To check the other homogeneities one can integrate by parts in the integral. As in our proof of Theorem~\ref{mel} this can be efficiently done by dilating the variables by means of a parameter $\lambda$. For example, making the dilation $z_j\mapsto\lambda z_j$ we get
$$\int_{\text{\rm Arg}^{-1}(\theta)}\frac{z^s}{\sum_k a_k z^{\alpha_k}}\frac{dz}{z}\ =\ \lambda^{s_j}\!\int_{\text{\rm Arg}^{-1}(\theta)}\frac{z^s}{\sum_k \lambda^{\alpha_{jk}}a_k z^{\alpha_k}}\frac{dz}{z}\,.$$
Differentiating both sides of this identity with respect to $\lambda$ and then putting $\lambda=1$, we find that
$$0=s_j\Phi+\sum_{k=1}^N\alpha_{jk}\,a_k\frac{\partial}{\partial a_k}\Phi\,,$$
and hence $E_j^\beta\Phi(a,s)=0$, with $\beta_j=-s_j$ as claimed.
\smallskip

We have thus established that $\Phi$ is an $A$-hypergeometric analytic function in the product domain $B(a)\times(\text{int}\,\Delta_f+i\,\mathbb{R}^n)$, and by uniqueness of analytic continuation its extension to the cylinder $B(a)\times\mathbb{C}^n$ will remain $A$-hypergeometric. Next, by the general theory of $A$-hypergeometric functions one has, for each fixed $s$, a (typically multivalued) analytic continuation of $a\mapsto\Phi(a,s)$ from $B(a)$ to all of $\mathbb{C}^N\setminus\Sigma_A$. Well known results on analytic functions of several variables then tell us that these continuations will still depend analytically on $s$, so we have achieved the desired analytic continuation to the full product domain $(\mathbb{C}^N\setminus\Sigma_A)\times\,\mathbb{C}^n$. The uniqueness of analytic continuation again guarantees that $\Phi$ will everywhere satisfy the $A$-hypergeometric system with the homogeneity parameter $(-1,-s_1.\ldots,-s_n)$.
\end{proof}
\smallskip

Related integral representations of $A$-hypergeometric functions have been considered by several authors, see for instance \cite{GKZb} and \cite{B}.
It is probably instructive to examine a concrete special instance of the above theorem, and we choose to present the case of the classical Gauss hypergeometric function.
\medskip

\noindent\textbf{Example 3.} Take $A=\{(0,0),(1,0),(0,1),(1,1)\}$ to consist of the four corners of the unit square in the first quadrant. It is easy to check that in this case 
$\Sigma_A$ is given by the equation $E_A(a)=a_1a_2a_3a_4(a_1a_4-a_2a_3)=0$.
Then consider the polynomial $f(z)=a_1+a_2z_1+a_2z_2+a_4z_1z_2$ together with its associated Mellin transform 
\begin{equation}\label{exint}M_{1/f}(a,s)=\int_0^\infty\!\!\!\!\int_0^\infty\frac{z_1^{s_1}z_2^{s_2}}{a_1+a_2z_1+a_3z_2+a_4z_1z_2}\frac{dz_1dz_2}{z_1z_2}\,.\end{equation}
Let us compute this transform for simplicity first in the case $a_1=a_2=a_3=1$. Writing $f$ as $(1+z_2)+(1+a_4z_2)z_1$, we can use formula (\ref{fraction}) and first perform the integration with respect to $z_1$. This yields the expression
$$\Gamma(s_1)\Gamma(1-s_1)\int_0^\infty(1+z_2)^{s_1-1}(1+a_4z_2)^{-s_1}z_2^{s_2-1}dz_2$$
for the Mellin transform (\ref{exint}). Re-writing the integrand and expanding in a power series we find that the above integral equals
$$\int_0^\infty\Bigl(1+\frac{(a_4-1)\,z_2}{1+z_2}\Bigr)^{-s_1}\frac{z_2^{s_2-1}dz_2}{1+z_2}=\sum_{k\ge0}\frac{\Gamma(1-s_1)}{\Gamma(1-s_1-k)\,k!}(a_4-1)^k\!\int_0^\infty\!\frac{z_2^{s_2+k-1}dz_2}{(1+z_2)^{1+k}},$$
so using the formula
$$\int_0^\infty\!\frac{w^{t}}{(1+w)^{1+k}}\frac{dw}{w}=\frac{1}{k!}\,\Gamma(t)\Gamma(1+k-t)$$
we find that (\ref{exint}) can be expressed as $\Gamma(s_1)\Gamma(1-s_1)\Gamma(s_2)\Gamma(1-s_2)$ times

$$\sum_{k\ge0}\frac{\Gamma(1-s_1)\Gamma(s_2+k)}{\Gamma(1-s_1-k)\Gamma(s_2)(k!)^2}(a_4-1)^k
=\sum_{k\ge0}\frac{\Gamma(s_1+k)\Gamma(s_2+k)}{\Gamma(s_1)\Gamma(s_2)(k!)^2}(1-a_4)^k\,.$$
\smallskip

\noindent In other words, we have shown that $\Phi(1,1,1,a_4,s_1,s_2)={}_2F_1(s_1,s_2;1;1-a_4)$. Using the homogeneities if $\Phi$, corresponding to the row vectors of the matrix 
$$A=\left(\begin{array}{ccccc}
1 & 1 & 1 & 1 \\
0 & 1 & 0 & 1 \\
0 & 0 & 1 & 1
\end{array}\right)$$
and the homogeneity parameter $(-1,-s_1,-s_2)$, we then easily recover the more general formula
\begin{equation}\label{fi}\Phi(a,s)= a_1^{s_1+s_2-1}a_2^{-s_1}a_3^{-s_2}\,{}_2F_1\bigl(s_1,s_2;1;1-\frac{a_1a_4}{a_2a_3}\bigr)\,.\end{equation}

\noindent Notice that the three singular points $0$, $1$ and $\infty$ for the Gauss function correspond to the factors $a_1a_4-a_2a_3$, $a_1a_4$ and $a_2a_3$ of the principal
$A$-determinant $E_A$.

\bigskip

\noindent The $A$-hypergeometric system consists in this case of the single binomial equation

$$\left(\frac{\partial^2}{\partial a_1\partial a_4}-\frac{\partial^2}{\partial a_2\partial a_3}\right)\Phi(a,s)=0\,,$$ 
\smallskip

\noindent together with the three homogeneity equations

$$\Biggl\{\begin{array}{rl}
(a_1\partial/\partial a_1+a_2\partial/\partial a_2+a_3\partial/\partial a_3+a_4\partial/\partial a_4+1)\,\Phi(a,s)&\!\!=\,\,0\,,\\
(a_2\partial/\partial a_2+a_4\partial/\partial a_4+s_1)\,\Phi(a,s)&\!\!=\,\,0\,,\\
(a_3\partial/\partial a_3+a_4\partial/\partial a_4+s_2)\,\Phi(a,s)&\!\!=\,\,0\,.
\end{array}$$
\smallskip

Let us end by considering what happens as one of the variables $a_k$ vanishes. From the Gauss hypergeometric theorem one knows that, for $\sigma_1+\sigma_2<1$, there is an identity ${}_2F_1\bigl(s_1,s_2;1;1)=\Gamma(1-s_1-s_2)/(\Gamma(1-s_1)\Gamma(1-s_2))$, so setting $a_4=0$ in the above formula (\ref{fi}) we get
$$\Phi(a_1,a_2,a_3,0,s_1,s_2)= a_1^{s_1+s_2-1}a_2^{-s_1}a_3^{-s_2}\frac{\Gamma(1-s_1-s_2)}{\Gamma(1-s_1)\Gamma(1-s_2)}\,.$$
This of course fits beautifully with the fact that the Mellin transform $M_{1/f}$ for the polynomial $f(z_1,z_2)=a_1+a_2z_1+a_3z_2$ is equal to
$$a_1^{s_1+s_2-1}a_2^{-s_1}a_3^{-s_2}\,\Gamma(s_1)\Gamma(s_2)\Gamma(1-s_1-s_2)\,,$$
compare with formula (\ref{fraction}) above.
\smallskip

To treat the case $a_1=0$ we can use Euler's hypergeometric transformation ${}_2F_1(a,b;1;z)=(1-z)^{1-a-b}{}_2F_1(1-a,1-b;1;z)$ and re-write formula (\ref{fi}) as
$$\Phi(a,s)= a_2^{s_2-1}a_3^{s_1-1}a_4^{1-s_1-s_2}\,{}_2F_1\bigl(1-s_1,1-s_2;1;1-\frac{a_1a_4}{a_2a_3}\bigr)\,.$$
Again using the hypergeometric theorem of Gauss, now for $(1-\sigma_1)+(1-\sigma_2)<1$ or equivalently $\sigma_1+\sigma_2>1$, we find that 
$$\Phi(0,a_2,a_3,a_4,s_1,s_2)= a_2^{s_2-1}a_3^{s_1-1}a_4^{1-s_1-s_2}\frac{\Gamma(s_1+s_2-1)}{\Gamma(s_1)\Gamma(s_2)}\,,$$
which can be seen to concord with the formula for the Mellin transform $M_{1/f}$ with $f(z_1,z_2)=a_2z_1+a_3z_2+a_4z_1z_2$.
Finally, the other two Euler transformations 
$${}_2F_1(a,b;1;z)=(1-z)^{-a}{}_2F_1(a,1-b;1;z/(z-1))=(1-z)^{-b}{}_2F_1(1-a,b;1;z/(z-1))$$
similarly lead to the formulas
$$\Phi(a_1,0,a_3,a_4,s_1,s_2)= a_1^{s_2-1}a_3^{s_1-s_2}a_4^{-s_1}\frac{\Gamma(s_2-s_1)}{\Gamma(1-s_1)\Gamma(s_2)}\,,$$
and
$$\Phi(a_1,a_2,0,a_4,s_1,s_2)= a_1^{s_1-1}a_2^{s_2-s_1}a_4^{-s_2}\frac{\Gamma(s_1-s_2)}{\Gamma(s_1)\Gamma(1-s_2)}\,.$$

\end{document}